\documentclass[10pt]{amsart}
\usepackage{amsmath, amssymb, amsthm}
\usepackage{tikz}
\usetikzlibrary{arrows.meta}
\usetikzlibrary{bayesnet}
\usepackage{booktabs}

\theoremstyle{plain}
\newtheorem{theorem}{Theorem}[section]
\newtheorem{definition}[theorem]{Definition}
\newtheorem{lemma}[theorem]{Lemma}
\newtheorem{remark}[theorem]{Remark}

\newcommand{\R}{\mathbb{R}}
\newcommand{\E}{\mathbb{E}}
\newcommand{\bp}{\mathbf{p}}
\newcommand{\bZ}{\mathbf{Z}}
\newcommand{\bX}{\mathbf{X}}
\newcommand{\bC}{\mathbf{C}}
\newcommand{\old}{\text{old}}
\newcommand{\new}{\text{new}}
\newcommand{\Yhat}{\widehat{Y}}
\newcommand{\Zhat}{\widehat{Z}}
\newcommand{\Uhat}{\widehat{U}}

\newcommand{\thetaold}{\theta^\old}
\newcommand{\balph}{\mathbf{\alpha}}

\DeclareMathOperator{\argmax}{argmax}
\DeclareMathOperator{\Dir}{Dir}

\title{Probabilistic Fitting of Topological Structure to Data}
\author{James T. Griffin}

\begin{document}

\begin{abstract}
  We define a class of probability distributions that we call \emph{simplicial mixture models}, inspired by simplicial complexes from algebraic topology.  The parameters of these distributions represent their topology and we show that it is possible and feasible to fit topological structure to data using a maximum-likelihood approach.
  We prove under reasonable assumptions that with a fixed number of vertices a distribution can be approximated arbitrarily closely by a simplicial mixture model when using enough simplices.
  Even if the topology is not of primary interest, when using a model that takes the topology of the data into account the vertex positions are good candidates for archetype/endmember vectors in unmixing problems.
\end{abstract}


\maketitle

\section{Introduction}
If it is reasonable to assume that a dataset can be generated by a simple process then it is appropriate to fit a simple model to that dataset, perhaps by using the mean, covariance matrix or other statistics.  However if the dataset is generated by a complicated process then it may have non-trivial topology.  Topological data analysis (TDA) provides tools for studying that topology~\cite{carlsson2009topology}, but as of yet there are no known methods for fitting topological structure to the distribution of the data.

The motivating problem we wish to solve is this: given observed data in $\R^n$ that we assume has been independently sampled from a distribution with inherent but unknown topology, how does one find a candidate topological structure modelling the data and how does one compare it with other candidates?
We wish to solve the problem in a probabilistic way: the model and its topology should be specified by the parameters of a distribution and by maximising the likelihood of the observed data over all possible parameters we should reveal the topology of the data.

There are many methods to fit geometric objects to data, for example a plane can be fitted to a cloud of data by total least squares minimisation, or data can be clustered around chosen points by fitting a Gaussian mixture model.  However many datasets exhibit more complicated topology not captured by these simple geometries.  This topology is studied in topological data analysis (TDA)~\cite{carlsson2009topology}, but current methods only estimate invariants of the topology.
More complicated objects can be fitted using a generative topographic mapping (GTM)~\cite{bishop1998gtm} which in theory could fit any choice of metric space, however this metric space must be specified in advance and the fitting will be poor unless the embedded space is initialised close to the final fit.
There appears to be a general principle at work, the more complicated the model, the more complicated the likelihood function on its parameter space and the more challenging it is to fit.

The \emph{simplicial mixture model} proposed in this paper can fit complicated geometry to data without the need to specify the topology in advance.  Moreover we observe that fitting with randomly initialised parameters is effective.  In algebraic topology the simplicial complex containing all possible simplices is homotopy equivalent to the simplest possible geometric object, a point.  By weighting the set of all possible simplices, and allowing simplices that do not represent the data to fade we gain the advantage of fitting a simple geometric object but the ability to express a wide range of topologies.

In TDA there are already methods of assigning simplicial complexes to data, most notably the witness complex~\cite{desilva2004topological}.  This relies on a scale parameter and the behaviour of its homological invariants as the scaling parameter changes defines the persistent homology.
However from our perspective this is not a fitting of a simplicial complex both because there is not a natural choice of scaling parameter and because it does not take into account the varying density of the distribution of the data.  Its purpose is to compute invariants of the support of the distribution, which is assumed to be some interesting subset of~$\R^n$.

In Section~\ref{section_smm} we give the full definition of a simplicial mixture model and in Section~\ref{section_applications} discuss two applications of our techniques to datasets: fitting topology to hand-written digits and then unmixing of an image into a number of channels.
In Section~\ref{section_degeneracy} we study distributions arising from simplices and prove our main Theorem: that simplicial mixture models can be used to approximate distributions on a convex hull of vertices arbitrarily closely.
We discuss further connections between our methods and other work in Section~\ref{section_discussion} and offer our conclusions in Section~\ref{section_final}.

In Appendix~\ref{section_LEMM} the mathematical underpinning of fitting a simplicial mixture model to observed data is presented.  It is convenient to study a more general class of model, linearly embedded mixture models. We derive an expectation-maximisation (EM) algorithm in Section~\ref{section_lemm_EM}.
In Appendix~\ref{section_mcmc} we discuss our implementation of a stochastic EM algorithm using Markov chains.

All of the code used to generate our results is available at~\cite{griffin2019smm}.  Readers are encouraged to download the code and experiment for themselves.

\section{Simplicial Mixture Models}
\label{section_smm}
To represent topology we use simplices, a generalisation of points, edges and triangles.  A \emph{geometric $k$-simplex} is the convex hull of $k+1$ points in general position, so a $0$-simplex is a point, a $1$-simplex is an edge, a $2$-simplex is a triangle, etc.

Roughly speaking, a simplicial mixture model is a weighted collection of simplices which have vertices from a common set $\{v_1,\cdots,v_m\}$ for some~$m$ and vectors $v_i\in\R^n$.  We now make the definition precise.  A \emph{combinatorial $k$-simplex} $S$ is a multi-set of $k+1$ vertices in $\{1,\cdots,m\}$.  By sorting the $k+1$ vertices a multi-set is represented by an increasing sequence $i_0\leq\cdots\leq i_k$. Alternatively, counting the number of each vertex yields non-negative integers $\balph = (\alpha_1,\cdots,\alpha_m)$ that sum to~$k+1$.
If any of the vertices are repeated we call the simplex $S$ \emph{degenerate}.  We call the set of vertices that occur at least once in $S$ the \emph{support of~$S$}. Write $A_k(m)$ for the set of all combinatorial $k$-simplices on $m$ vertices.


Each simplex $S=(i_0\leq\ldots\leq i_k)$ defines a random vector $U_S\in\R^m$ by sampling uniformly from the standard probability simplex $\Delta^k \subset \R^{k+1}$, then applying the linear map sending $e_j\in\R^{k+1}$ to $e_{i_j}$, we call the resulting random vector $U_S$.
To sample uniformly from the standard simplex one can take $k+1$ independent samples from the exponential distribution, $\mu_0,\cdots,\mu_k$ then normalise the results by letting $u_i = \mu_i / \sum_{j=0}^{k+1}\mu_j$.
Alternatively one can sample $k$ points $\nu_1,\ldots,\nu_k$ from the uniform distribution on $[0,1]$, sort them, pad the now increasing sequence at each end by~0 and~1 respectively and then take the differences between neighbouring entries in the sequence. Using either method gives a uniform distribution on the simplex~$\Delta^k$.

\begin{definition}\label{def_smm}
Let $V$ be an $n\times m$ matrix and $\bp$ be a probability distribution on~$A_k(m)$. The \emph{simplicial mixture model} with parameters $(\bp, V)$ is the random vector in $\R^n$ defined by
\begin{equation}\label{eq_XeqVUC}
  X = VZ = VU_C,
\end{equation}
where $C$ is the discrete random variable on $A_k(m)$ defined by $\bp$ and where $Z=U_C$ is the mixture of random vectors $U_S\in\R^m$ indexed by~$C$.
\end{definition}

\begin{remark}
The choice to use the set $A_k(m)$ of simplices is driven by the theory presented in Section~\ref{section_degeneracy}.  In particular, even though $A_k(m)$ and $A_{k+1}(m)$ are disjoint sets, the family of models for $A_k(m)$ is a proper subset of the family of models based on~$A_{k+1}(m)$.
\end{remark}

\begin{figure}
  \begin{centering}
  \includegraphics{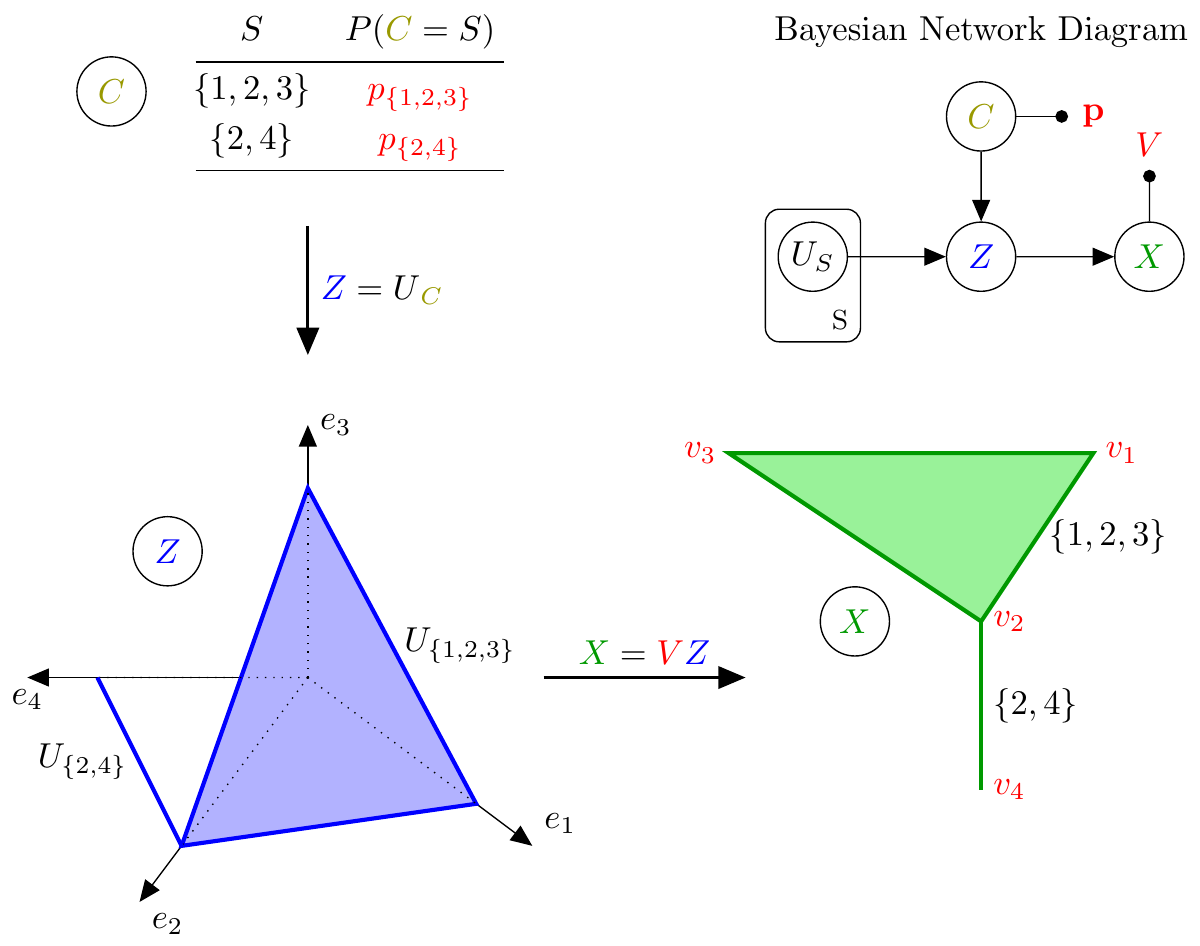}\\
  \end{centering}
  \definecolor{mustard}{rgb}{0.6,0.6,0.0}
  \definecolor{mygreen}{rgb}{0.0,0.6,0.0}
  \caption{A diagram illustrating a simplicial mixture model with $m=4$ vertices in $n=2$ dimensions supported on two simplices.  The parameters are in \textcolor{red}{red}.  Top right: The Bayesian network shows the dependency of the random variables and parameters for a general simplicial mixture model.  Top left: a table describing $\textcolor{mustard}{C}$, a discrete latent variable with support the two combinatorial simplices.  Bottom left: $\textcolor{blue}{Z}$ is a continuous latent variable in~$\R^m$.  Bottom right: the output variable $\textcolor{mygreen}{X}$~in $\R^n$.}
  \label{fig_intro}
\end{figure}

It does not make sense to directly apply a maximum likelihood approach to fitting the parameters $(\bp,V)$ to observed data because $X$ does not necessarily have a probability density function. So to fit the parameters $(\bp,V)$ we introduce a new random vector $X + N(0, \Sigma)$ for some covariance matrix $\Sigma$, then maximise the log-likelihood over the trio~$(\bp, V, \Sigma)$.
Of course this is not possible to maximise directly, but an iterative approach is possible using the expectation-maximisation~(EM) algorithm.  The algorithm and its derivation is presented in Appendix~\ref{section_lemm_EM}.  However there is a complication: for simplices of dimension~2 or greater the expected values required in the EM~algorithm can not be calculated directly, so must be estimated; a stochastic EM algorithm is presented in Appendix~\ref{section_mcmc}.


For models with only 1-dimensional simplices the log-likelihood of parameters can be used to judge the fit of a simplicial mixture model to data.  However this will typically favour complicated models over simpler models.  It is also not always feasible to compute or even estimate the log-likelihood when the simplices are of dimension greater than one.
As an alternative to the log-likelihood we suggest the \emph{intrinsic encoding rate},
\begin{equation}\label{eq_encodingrate}
  h_R(\bp,V,\Sigma) = H(\bp) +
    \sum_{S\in A} p_S R_{VU_S}(\Sigma) + \frac12\log\left[(4\pi e)^n\det(\Sigma)\right],
\end{equation}
which can be used to measure the fit of a model when the parameters $(\bp,V,\Sigma)$ are yielded by the EM algorithm.  A lower rate is associated to a better fit.
Here $H(\bp)$ is the entropy of the distribution~$\bp$, the expressions $R_{VU_S}(\Sigma)$ are rate distortion functions associated to each simplex $S$ and the final expression is the differential entropy of the multivariate normal distribution~$N(0,2\Sigma)$.

The intrinsic encoding rate is a measure of the amount of information required to encode a typical vector $x\sim X + N(0,\Sigma) = VU_C + N(0,\Sigma)$ by a supplying a triple $(S, z, x)$.  The entropy term is the encoding rate for~$S\in C$; the second term is the expected amount of information required to express $z\in Z$ given $S\in C$ such that $VZ$ has an accuracy depending on~$\Sigma$; and the final term is the amount of information required to encode the difference $x - Vz$ under the hypothesis that $x-Vz$ is distributed as a normal distribution with mean~$0$ and covariance~$2\Sigma$.
As the latent variables must be encoded the intrinsic encoding rate favours simpler models.  And as the differences between the data and the latent variable~$X-VZ$ must be encoded the rate favours small covariances~$\Sigma$.
The derivation will be explained further in Appendix~\ref{section_encodingrate}.

\section{Example applications}
\label{section_applications}
To demonstrate that the fitting of simplicial mixture models behaves as one might intuitively expect we present an example of fitting models to hand-written digits.  Then for a single digit we observe the range of outputs of the fitting algorithm with different random initialisations of the parameters.

Following this we demonstrate an application to image analysis.  A simplicial mixture model can take advantage to the geometry and topology of the distribution of colours in an image to unmix the image into multiple channels, or layers.  In this case the topology of the space of colours is only of secondary interest, but the fitted vertex positions (in this case colours) form a palette and for each pixel the conditional distribution of the latent variable~$Z$ describes the mixing of the palette colours.
The geometry of the space of colours was used in a similar way in~\cite{tan2017decomposing}.

\subsection{Fitting models to hand-written digits}
\label{section_MNIST}
\begin{figure}
  \begin{centering}
  \includegraphics{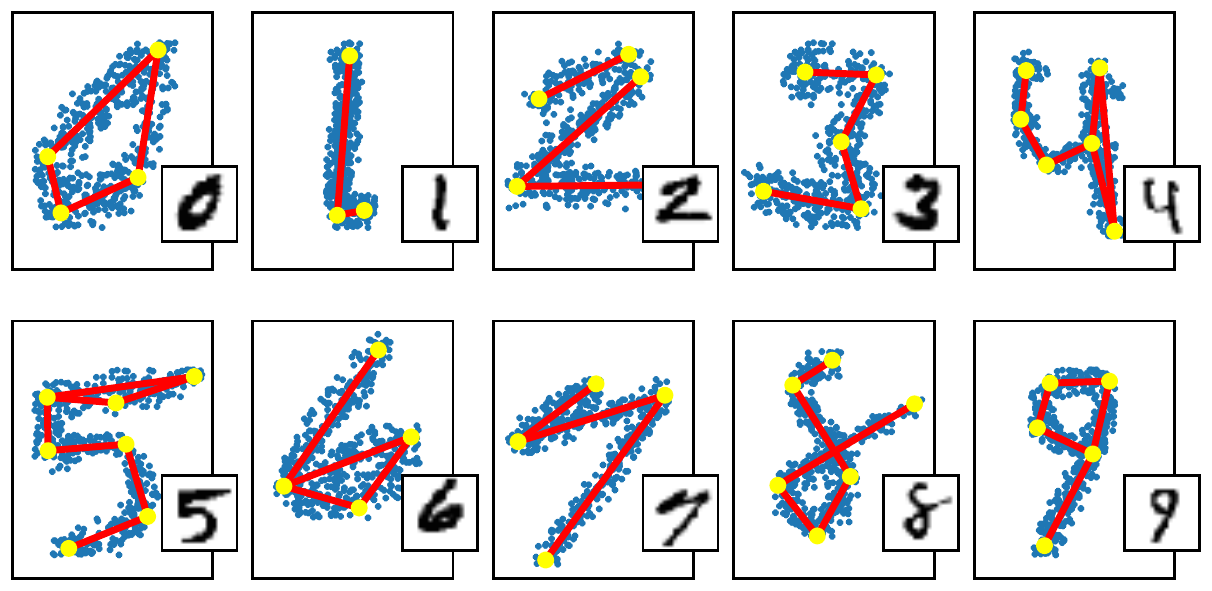}\\
  \end{centering}
  \caption{The result of fitting a simplicial mixture model to each of 10 digits taken from the MNIST database~\cite{mnist}.  The 1-simplices (edges) with 5\% or greater of the probability mass are plotted in red, the vertex positions are plotted in yellow and the blue points are random samples from the inset raster image.  Observe that for the most part the lines follow an angular fit to the stroke of the pen.
  The details of how the models were fitted are given in Section~\ref{section_MNIST}.}
  \label{fig_10digits}
\end{figure}

A hand-written digit has a simple topology determined by the path of the pen as it was drawn, so it is natural to fit a simplicial mixture model using the set of 1-simplices~$A_1(m)$, see~\cite{hinton1992adaptive} for a similar approach using a single spline.
We fit a number of models for different values of $m$.
The results of treating a grayscale image as a distribution and then fitting a simplicial mixture model to this distribution are presented in Figure~\ref{fig_10digits}.
The fitted simplices follow the strokes of the pen as one would expect.

From each image 500 samples were taken, each sampled by choosing a pixel with probability propositional to its intensity, then sampling uniformly from the bounds of the pixel.
For each of the ten digits $d$ and each choice of $m=3,\cdots, 7$, the vertex positions were initialised by picking $m$ sample points at random from the distribution. This was repeated 10 times for each~$(d,m)$.
Then 40 steps of the EM algorithm were applied to each initialisation to give 10 different candidate parameters for each~$(d,m)$ and the parameters with the lowest intrinsic encoding rate were chosen for each pair.  An additional 500 steps of the EM algorithm were applied to give an optimised model for each pair~$(d,m)$ and then for each $d$ the model with the smallest intrinsic encoding rate was chosen.

\begin{figure}
  \begin{centering}
  \includegraphics{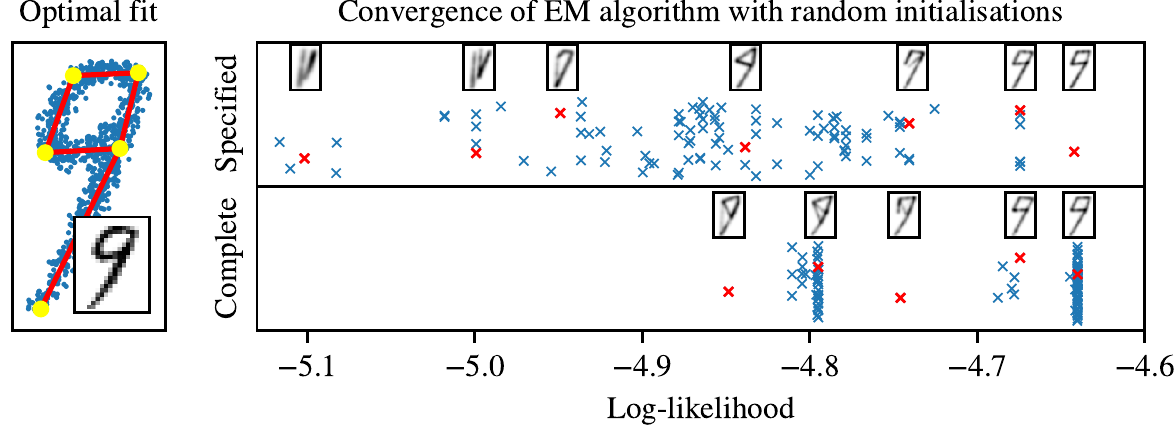}\\
  \end{centering}
  \caption{
  We demonstrate the range of results for different initialisations of the EM algorithm.
  Left: the model with the greatest log-likelihood amongst all fittings with $m=5$ vertices is shown in red, the blue points are the samples from the image used to fit the model.  The inset is the original hand-written digit.
  Right: the log-likelihoods are plotted along the $x$-axis for the results of applying the EM algorithm for 200 random initialisations of the vertex positions.  To aid in visualisation the $y$-values are randomly offset.
  The top half of the plot shows 100 initialisations where the optimal graph structure is pre-defined. The bottom half shows initialisations with all possible edges. The insets show the probability density functions for a sample of the models.
  For the complete graph there appear to be two prominent local maxima and out of the 100 initialisations,~33 find themselves in the range $-4.8\pm 0.02$ and~59 find themselves in the range $-4.64\pm 0.02$.
  For the pre-defined graph there are many more local maxima and the values of the log-likelihoods are spread out in a large range.}
  \label{fig_insideno9}
\end{figure}

For each of the 10 digits we plotted only the best fit from multiple initialisations as judged by the intrinsic encoding rate. We now focus on a single digit, fix the number of vertices at $m=5$, then investigate the range of results obtained with different random initialisations.
We also investigate the difference between starting with a specified set of edges or starting without any prior knowledge with the complete graph, by fitting models of both types to the chosen digit.
Since the number of vertices is constant we use the log-likelihood to judge the quality of fitting, see Figure~\ref{fig_insideno9} for the results.

For the specified model we observe many different values of the log-likelihood in a large range corresponding to many local maxima of differing quality.  The range of quality is not altogether surprising because although we specified the edges of the graph we randomly initialise the vertex positions.
For the unspecified model with the complete graph the results were very different with two main values observed.  In the optimum fit four of the vertices are used to describe the loop of the `9' with a single leaf edge describing the tail.  In the secondary fit only three vertices describe the loop while the extra vertex is used to better represent the curve of the tail.  There are further values obtained but the regions around the two most prominent values accounts for over 90 of the 100 runs.

The details of how the figure was generated are as follows.
For $m=5$ vertices the optimum graph fitting the chosen digit `9' consists of a square with an additional leaf attached.
For both this graph and the complete graph we used the EM algorithm to fit a simplicial mixture model with 100 different random initialisations.
For each run 500 steps of the algorithm were used and the resulting log-likelihoods were plotted to give an indication of the distribution of local maxima.
With the complete graph the log-likelihoods occur around relatively few values, with 59 out of 100 occuring in a small range of the maximum observed log-likelihood.
Whereas for the pre-defined graph only one out of 100 was in this range and many local maxima were found for a larger range of values.

\subsection{Unmixing of images}
\label{section_endmember}
By making an assumption on the data we can use a simplicial mixture model to unmix the data into a linear combination of vertex positions.
We assume that for a small set of unknown vectors the data are formed taking a linear mixture of a subset of those vectors.  The coefficients of that mixture should sum to~1, so the data is in the convex hull of the subset of vectors.
In the case of an image this assumes that there is a palette of colours and that each pixel is formed by combining some colours from the palette.

In spectral unmixing as applied in geology~\cite{bioucas2012hyperspectral}, the spectrum of a geological sample is assumed to be a mixture of spectra of pure minerals.  These minerals are called endmembers.  When the spectra of the endmembers are unknown the problem of infering the endmembers from sample spectra is called endmember extraction.
A similar approach is taken in archetypal analysis~\cite{cutler1994archetypal}, where data are again assumed to be linear combinations of certain archetypes.  However in archetype analysis the archetypes themselves are assumed to be linear combinations of observed data.

In a simplicial mixture model the variable $X=VZ$ is a mixture of vertex positions $V$ using the latent variable $Z$ as mixing coefficients.
Hence by fitting a simplicial mixture model to data the vertex positions $V$ give candidate endmember vectors, whilst for each datapoint the estimated latent variable $Z$ gives mixing coefficents between those endmembers.
In applications it is reasonable to assume that the distribution of the mixing values is not uniform, perhaps due to a physical reason why two endmembers cannot coexist, or perhaps two features can occur individually but are more likely to coexist.
So we expect that using a model that allowed for highly non-uniform distributions would perform well in both finding endmembers/archetypes and the associated mixing coefficients.

Usually the maximum number of endmembers is taken to be one greater than the number of components of the data so that the function mapping the simplex of mixing coefficients onto the convex hull of the endmembers is injective allowing the mixing coefficients to be uniquely determined.  However if a highly non-uniform distribution is modelled then a Bayesian approach can infer the mixing probabilities even when the number of features is two or more larger than the dimension of the data.
A geometric way to interpret this is that data can have a complicated geometry which a simplicial mixture model can fit.  The latent variable $Z$ specifies the position of data within this geometry as a combination of only a few feature vectors and hence as a sparse vector with a large number of zero components.

We tested this use of a simplicial mixture model by fitting a model to an image, viewing the RGB values of each pixel as a datapoint.
In Figure~\ref{fig_unmixing}, an image of an iris plant is unmixed from its original three RGB channels to seven channels.  The simplicial mixture model uses the set $A_3(7)$ of all possible 3-simplices (tetrahedra).

Observe in particular that channel~6 corresponds to the pigment of the petals, channel~4 to light that has scattered through the plant while channel~2 is pure white where the image is oversaturated.

\begin{figure}
  \begin{centering}
  \includegraphics{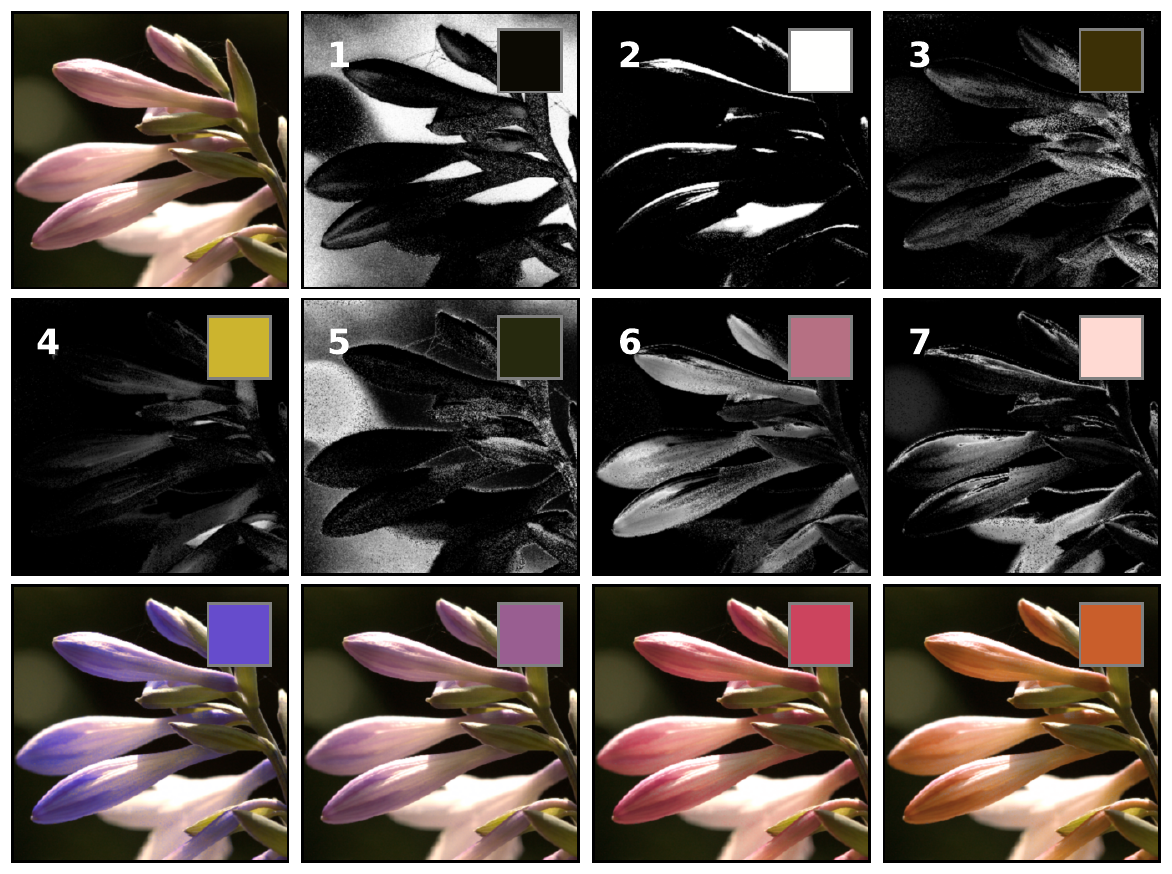}\\
  \end{centering}
  \caption{An image (top-left) is unmixed into 7 channels (numbered) demonstrating the method of using a simplicial mixture model to take advantage of the geometry of a dataset.
  In the bottom row the original image is recoloured by changing channel~6.
  The simplicial mixture model is based on $A_3(7)$ the set of all possible 3-simplices on 7 vectors and then taking the expected value of the conditional distribution of the latent variable~$Z$ given each pixel value.  Each of the 7 plots is the intensity of a single channel. The coloured box for each channel represents the colour for that channel. The original image was taken from the MIT-Adobe database~\cite{MIT2011fivek}.}
  \label{fig_unmixing}
\end{figure}

Since the dimension of the simplices is greater than~1 it is not possible to apply the EM algorithm directly and we used a Markov chain Monte Carlo based stochastic version of the EM algorithm.  A set of 32041 pixels were taken from the image to fit the model.  The vertex positions were initialised at random by sampling from the pixels, however to ensure that the 7 samples were spread out they were iteratively resampled. We chose and removed one of the 7 at random, then from 100 candidate pixels chose the pixel furthest away from the remaining 6 pixels.  This was repeated 100 times.
From this initialisation we used the stochastic EM algorithm with a total of 3000 maximisations.  The full description of the algorithm is given by expression~\eqref{eq_regime} after the stochastic algorithm is introduced in Appendix~\ref{section_mcmc}.

Remarkably even though each pixel is specified by only three coordinates its position within the geometry of the cloud of pixels has allowed it to be unmixed to seven coordinates.  The representation by these seven numbers will typically be sparse, meaning that all but a few will be close to zero.

\section{The approximation theorem}
\label{section_degeneracy}
Theorem~\ref{thm_approximation} states that ``any bounded distribution on $\R^n$ can be approximated arbitrarily closely by a simplicial mixture model with~$n+1$ vertices.''  In this section we study the underlying density functions associated to the distributions $U_S$ and their mixtures.  We then prove Lemma~\ref{lem_spanproperties} which justifies the choice of $A_k(m)$ to index mixtures of simplices.  Finally we describe a form of kernel density estimation on the simplex which we use to prove the approximation theorem.

Intuitively a simplicial mixture model allows the fitting of discrete combinatorial objects constructed from simplices, as shown in Figures~\ref{fig_intro} and~\ref{fig_10digits}.  Under this intuition, to closely represent an arbitrary distribution would require many vertices, using them to discretise the distribution of the data.  However the intuition is incomplete when degenerate simplices are included.
The degenerate simplices can be to perform kernel density estimation.  As the dimension of the degenerate simplices grows, the kernels get smaller and the approximation closer.

A degenerate $k$-simplex embedded in $\R^n$ is much like a regular simplex, but where some of the vertices coincide.
In the extreme case, the $k$-simplex $(i\leq \cdots \leq i)$ defines a constant random variable at vertex~$i$.
The 2-simplex $(1 \leq 1 \leq 2)$ is a triangle but where two vertices coincide. For intuition consider a non-degenerate triangle where the vertices $v_1,v_2,v_3\in\R^2$ are in general position, but vertices $v_1$ and $v_2$ are very close.  The associated random variable is uniform on the triangle, but since $v_1$ and $v_2$ are close it looks much like a thickened line where the thickness of the line varies linearly down to zero at~$v_3$.
So when the points $v_1$ and $v_2$ coincide the result is a distribution with support the line segment joining $v_1=v_2$ and $v_3$, but with the density on the line varying linearly; a triangular distribution.

We now consider a general $k$-simplex $S$. For convenience we assume without loss of generality that the vertex set is $\{0,\cdots,m\}$ with labelling starting at~0 and that the support is this whole set, if not then restrict to the subspace of $\R^{m+1}$ spanned by vectors in the support. Under this assumption the distribution $U_S$ is the well known Dirichlet distribution with parameters $\balph=(\alpha_0,\cdots,\alpha_m)$, where $\alpha_j$ is the number of vertices in $S$ equal to $j$, which by the assumption on the support is strictly greater than~$0$.
This is proved by first noting that from the definition the non-degenerate simplex $(1,\cdots,1)$ gives the uniform distribution on the simplex, which agrees with the corresponding Dirichlet distribution.  The result for degenerate simplices then follows after noting the following classical property of Dirichlet distributions, if you sample a point $(p_0,\cdots,p_m)$ from $\Dir(\alpha_0,\cdots,\alpha_m)$ then sum the first two coordinates, the result $(p_0+p_1,p_2,\cdots,p_m)$ is distributed according to $\Dir(\alpha_0+\alpha_1,\alpha_2,\cdots,\alpha_m)$.

We assume the natural Lesbegue measure on the simplex $\Delta^m\subset\R^{m+1}$, but suitably scaled so that the volume of the simplex is~$\tfrac1{m!}$. Then the density function of $U_S$ is expressed via a monomial,
\begin{equation}\label{eq_USrho}
  \rho_{U_S}(z_0,\cdots,z_m) = \frac1{B(\balph)} \prod_{j=0}^m z_j^{\alpha_j-1}.
\end{equation}
Here $B$ is the multivariate beta function,
\begin{equation}
B(\alpha) = \frac{\prod_{j=0}^m\Gamma(\alpha_j)}{\Gamma(\sum_{j=0}^m\alpha_j)}
\end{equation}
and $\Gamma(n) = (n-1)!$.
Hence the $k$-simplices with support $\{0,\cdots,m\}$ are in bijection with the degree $k-m$ monomials in $z_0,\cdots,z_m$.

The functions we consider have support on the simplex $\Delta^m$, which means that $z_0 +\cdots+z_m = 1$.  So the algebra of polynomial functions on $\Delta^m$ is the quotient ring
\begin{equation}
R_m = \R[z_0,\cdots,z_{m}] / (z_0+\cdots+z_m-1).
\end{equation}
Write $R_m^+$ for the convex subset of elements which may be expressed as polynomials with non-negative coefficients which integrate to~1 on $\Delta^m$, this parametrises the finite mixtures of distributions $U_S$ for simplices $S$ with full support.
The ring $R_m$ is isomorphic to the polynomial ring $\R[z_0,\cdots,z_{m-1}]$ via the formula $z_m = 1-z_0-\cdots-z_{m-1}$. In particular the formula $z_0+\cdots+z_m = 1$ implies that the mixture of the $m+1$ degenerate $(m+1)$-simplices is the uniform distribution.  So the probability density functions of the simplices with a fixed support are not linearly independent.
Fortunately the linear dependence is very simple and has some convenient benefits.
Let $B_k(m+1)\subset A_k(m+1)$ be the subset of $k$-simplices with support $\{0,\cdots,m\}$ and write $R_{m,l}$ for the span of degree $l$ monomials in $R_m$, i.e. the span of the distribution functions of simplices in~$B_{k-m}(m)$.
\begin{lemma}
\label{lem_spanproperties}
\begin{enumerate}
  \item
    For every $k$-simplex $S$ and $k'>k$, the density function $\rho_{U_S}$ is equal to a mixture of density functions $\rho_{U_{S'}}$ for $k'$-simplices $S'$.  Hence $R_{m,l} \leq R_{m,l'}$ for $l\leq l'$.
  \item
    The density functions $(\rho_{U_S})_{S\in B_k(m)}$ are linearly independent, so form a basis for~$R_{m,k-m}$.
\end{enumerate}
\end{lemma}
\begin{proof}
  For part 1), multiply the monomial $\prod_{j=0}^m z_j^{\alpha_j-1}$ associated to~$S$ by
  \begin{equation}
  (z_0+\cdots+z_m)^{k'-k}
  \end{equation}
  and expand to give a sum involving monomials corresponding to $k'$-simplices.   Within the algebra of functions on $\Delta^m$ this is the same as multiplying by~1. So we have an equation, equating the original monomial with a sum of monomials.  By rescaling the coefficients this is easily rewritten as an equation involving the density functions~\eqref{eq_USrho}.

  For the linear independence, observe that under the isomorphism of $R_m$ with $\R[z_0,\cdots,z_{m-1}]$ the span $R_{m,k-m}$ of the monomials of total degree $k-m$ in $R_m$ is isomorphic to the subspace of polynomials in $z_0,\cdots,z_{m-1}$ of degree less than or equal to $k-m$. But the monomials in $z_0,\cdots,z_{m-1}$ of degree less than or equal to $k-m$ are in bijection with the monomials in $z_0,\cdots,z_m$ of degree equal to $k-m$, so the monomials are linearly independent in $R_m$.
\end{proof}
These two properties mean that there is a infinite sequence of inclusions, i.e. a filtration
\begin{equation}
R_{m,0} \leq R_{m,1}\leq R_{m,2} \leq \cdots
\end{equation}
whose union is $R_m$ itself.  The significance of this filtration is that as $k$ is increased, the family of distributions of mixtures of $k$-simplices are nested within each other.  The same comment applies to simplicial mixture models.  So by fitting simplicial mixture models with successively higher dimensional simplices one gains successively better approximations to a true distribution.

To prove Theorem~\ref{thm_approximation} we will describe a form of kernel density estimation on the simplex~$\Delta^m$.
Suppose that $Y$ is a distribution on $\Delta^m$ and let $l\geq 1$ be a natural number.  Consider the following Markov chain
\begin{equation}\label{eq_KMarkov}
\begin{tikzpicture}
  %
  \node[latent] (Y) at (0,0) {$Y$};
  \node[latent] (A) at (4,0) {$A$};
  \node[latent] (Yh) at (8,0) {$\Yhat_l$};

  \path (Y) edge[->] node[above] {$\alpha\sim\text{MulNom}(l,y)$} (A);
  \path (A) edge[->] node[above] {$\widehat{y}\sim \text{Dir}(\alpha+1)$} (Yh);
\end{tikzpicture}
\end{equation}
where $A$ is given by drawing a sample $\alpha=(\alpha_0,\cdots,\alpha_m)$ from the multinomial distribution with multiplicity parameter~$l$ and distribution parameter $y$ drawn from~$Y$.  Then given $\alpha$ from $A$, $\Yhat_l$ is drawn from the Dirichlet distribution with parameters $(\alpha_j+1)_{j=0}^m$.

The output $\Yhat_l$ is a mixture of integer-valued Dirichlet distributions, whose parameters sum to $l+m$, so $\Yhat_l$ should be thought of as an approximation to $Y$ belonging to a finite dimensional family of distributions.
The mixture $\Yhat_l$ may be thought of as the result of applying a kernel density estimate to $Y$ where information is lost by sampling~$l$ points from each finite distribution $y\sim Y$ and then infering a posterior distribution for $y$ with the prior of a Dirichlet distribution $\text{Dir}(1,\cdots,1)$.
By increasing $l$ less information is lost.

For each $l$ define a kernel by marginalising out~$A$,
\begin{equation}
K_l(y, \widehat{y}) = P(\widehat{y}\mid y) =
  \sum_{\balph\mid \sum_{j=0}^m\alpha_j = l}
  \binom{l}{\alpha} \prod_{j=0}^my_j^{\alpha_j} \,\,\,
  \frac1{B(\balph+1)} \prod_{j=0}^m \widehat{y}_j^{\alpha_j}
\end{equation}
This is a symmetric function.  For any distribution $(Y,\mu)$ on $\Delta^m$ the distribution $\Yhat_l$ has the probability density function
\begin{equation}
\rho_{Y,l}(\widehat{y}) =
  \int_{\Delta^l} K_l(\widehat{y},y)d\mu(y) =
  \sum_\alpha \binom{l}{\alpha}\E_Y \Bigl(\prod_{j=0}^my_j^{\alpha_j}\Bigr)
  \frac1{B(\alpha+1)} \prod_{j=0}^m \widehat{y}_j^{\alpha_j}.
\end{equation}
\begin{theorem}\label{thm_approximation}
    The sequence $\Yhat_l$ converges weakly to the original distribution~$Y$.
    In particular the space $R^+_m$ of mixtures of Dirichlet distributions with integer parameters is dense in the space of all distributions on $\Delta^m$ with respect to weak convergence.

    Therefore any bounded distribution on $\R^n$ can be approximated arbitrarily closely by a simplicial mixture model with~$n+1$ vertices.
\end{theorem}
\begin{proof}
  If is sufficient to prove the first part for $Y$ a constant point distribution, i.e. that for a fixed value $y=(y_j)\in\Delta^m$, the kernels $\Yhat_l = K_l(y, \widehat{y})$ converge weakly to the constant distribution at~$y$.  So it is sufficient to show that the sequence of means tends to $y$ and that the sequence of covariance matrices tends to~$0$.
  This follows from direct calculation, for the $i$th coordinate of the mean of~$\Yhat_l$ we have
  \begin{align}
    \E \Yhat_{l,i}
    &= \sum_\alpha \binom{l}{\alpha} \prod_j y_i^{\alpha_i} \E \Dir(\alpha+1)
     = \sum_\alpha \binom{l}{\alpha} \prod_j y_i^{\alpha_i} \frac{\alpha_i+1}{l+m}\\
    &= \frac1{l+m} \frac{d}{dy_i}\Bigl[ y_i \bigl(\sum_i y_i\bigr)^l\Bigr]
     = \frac1{l+m} \bigl(\sum_iy_i\bigr)^l + \frac{l}{l+m} y_i\bigl(\sum_iy_i\bigr)^{l-1}\\
    &= \frac1{l+m}(ly_i + 1),
  \end{align}
  so the means converge to $y$ as $l\rightarrow\infty$.
  The approximation $\Yhat_l$ is a mixture of distributions and the covariance matrix of a mixture is the covariance of the means plus the mean of the covariances.
  The magnitude of each entry of the covariance matrix of a Dirichlet distribution whose parameters sum to $l+m$ is bounded above by $\tfrac1{l+m}$, and hence this applies to the mixture of the covariance matrices too.  So it remains to show that the covariance matrix of the means tends to~$0$.
  The $ij$th entry is
  \begin{equation}
    \text{Cov}(\E\Dir(\alpha+1)_i,\E\Dir(\alpha+1)_j)
    = \sum_\alpha\binom{l}{\alpha}\prod_i y_i^{\alpha_i}
      \frac1{(l+m)^2}(\alpha_i - ly_i)(\alpha_j - ly_j)
  \end{equation}
  which tends to~$0$ by standard properties of the multinomial distribution.

  The final part of the Theorem follows immediately after noting that if a distribution $X$ on $\R^n$ has bounded support then one can choose $n+1$ vertex positions such that the support of $X$ is contained in the convex hull of the vertices.  Let the vertex vectors be arranged as the columns of a matrix~$V$.  This allows the distribution to be pulled back, $X=VY$ to a distribution $Y$ on the standard $n$-simplex.  Then the sequence $\Yhat_l$ converges weakly to~$Y$ and the sequence of simplicial mixture models $V\Yhat_l$ converges weakly to $VY=X$.
\end{proof}
If we were to restrict to non-degenerate simplices there would be a total of $2^{n+1}$ simplices, so the possible distributions are limited, hence the degenerate simplices are required.

Although this result shows that distributions can be well represented by mixtures of simplex distributions with large enough simplex dimensions, the vertex positions are not being used.
Simplicial mixture models combine the fitting of vertex positions to represent important points in geometry along with the ability to approximate distributions in their convex hull.

\section{Discussion}
\label{section_discussion}
Now that we have given example applications of simplicial mixture models and have explained some of their theoretical properties we will discuss some links with existing methods from the literature.

A generative topographic mapping (GTM)~\cite{bishop1998gtm} describes the fitting of a geometric object to data using an expectation-maximisation (EM) algorithm.  The geometric object must be defined in advance, whereas our topological objects do not need to be pre-defined.
However the factorisation of the mapping into latent variables and a linear projection is common to both methods and the theory in Appendix~\ref{section_LEMM} also applies to GTMs.

We discussed endmember analysis in Section~\ref{section_endmember}.
The fitting of a single simplex to data has been approached from many angles, see~\cite{bioucas2012hyperspectral} for a review.
In particular the Bayesian approach~\cite{parra2000unmixing} is similar to our methods when restricting to a single simplex with a uniform distribution.

In non-negative matrix factorisation (NMF) a matrix of data is factorised into the product of a typically sparse matrix of non-negative values and another matrix~\cite{hoyer2004non}.  In our model we similarly have $X=VZ$, where $Z$ will be sparse.  The main difference is that a simplicial mixture model inherently takes into account the distribution of the set of supports of the sparse vectors.  Of course this comes at a computational cost.

Methods from algebraic topology have been applied in data analysis~\cite{carlsson2009topology}, the main methods are persistent homology and the mapper algorithm.  These methods compute invariants associated to the data: homology and the Reeb graph respectively.
We may hope that a simplicial mixture model captures a representation of the topology of a dataset.  In the case of the hand-written digits we plotted the 1-simplices with over~5\% of the probability mass and obtained embedded graphs representing the digits.
However this not only excludes weak simplices but also small simplices.  For more complicated examples, in particular those with higher dimensional and degenerate simplices, the challenge of extracting combinatorial information from the probabilities~$\bp$ on the simplices will be harder.

A close look at Figure~\ref{fig_10digits} reveals another issue, consider the figure `8'; the fitted model is actually a chain of line segments, so topologically a line.  The embedding defined by the vertex positions causes the line to follow the stroke of the hand-written digit.  Taking only the combinatorial topology we have not found the topology of the digit, but instead of a line parametrising the digit.
The combinatorial topology within a simplicial mixture model is encoded within the latent variable~$C$ and the geometric version of this topology is associated with the latent variable~$Z\in\R^m$.
This is only related to the original dataset by the linear map~$V:\R^m\rightarrow\R^n$.

\section{Conclusions}
\label{section_final}
In this article we have introduced a probabilistic method to fit topological structures to data and via two examples have demonstrated its viability.
We have proved the theoretical result that simplicial mixture models can approximate bounded distributions arbitrarily closely.

We have also argued that these methods could be of practical use outside of the field of topological data analysis, in particular to problems studied in endmember and archetype analysis.
Future work should describe applications to higher dimensional real-world datasets alongside a comparison to established methods.

Also required is further study into the properties of fitted models, in particular whether the method of fitting a simplicial mixture model to simulated or well-understood data reliably recovers known ground truths.


From the perspective of algebraic topology there are questions to answer: degenerate simplices are important, but do face maps have a place in the theory?  Potentially related, how does one compare two different models fitted to the same data?  These models were inspired by algebraic topology, so what do homotopies between simplicial mixture models look like and can they be used to compare models?

A final observation.  There is a vector space structure on the probability simplex; this is given by identifying the simplex with a vector space via the logistic map~\cite{aitchison1980logistic}.  The space of linear maps $\R^m\rightarrow\R^n$ is also a vector space and hence the set of parameters $(\bp, V)$ of a simplicial mixture model is itself a vector space.
This means that one can define a simplicial mixture model to take values in the parameter space of a different simplicial mixture model.  For example one could use such a hierarchical model to represent not just a single digit of the MNIST dataset, but the space of all digits.  This applies generally whenever we are presented with a dataset of datasets.

\bibliographystyle{plain}
\bibliography{library}

\appendix

\section{Linearly Embedded Mixture Models}
\label{section_LEMM}
A linearly embedded mixture model is a convenient abstraction of a simplicial mixture model.  We will derive the expectation-maximisation (EM) algorithm and the intrinsic encoding rate in this generality.

The abstraction takes advantage of the fact that the central formula $X=VU_C$ in Definition~\ref{def_smm} could apply to any family of random vectors $U_S\in\R^m$ for $S$ in an indexing set $A$.
Given a distribution $\bp$ on $A$ and a linear map $V:\R^m\rightarrow \R^n$ the \emph{linearly embedded mixture model} parametrised by $(\bp,V)$ is defined by the formula $X=VU_C$.  With $Z=U_C \in \R^m$ a mixture of distributions and $X=VZ$ it is literally a linearly embedded mixture of the distributions~$U_S$.
The simplicity of the model means that existence of simple statistics and their formulae depend on existence and formulae for each $U_S$, at its most simple we have the formula
\begin{equation}
\E X = V\E Z = \sum_{S\in A} p_S V\E U_S
\end{equation}
and the formulae for higher central moments are also straightforward.

To provide a model for observed data we choose a covariance matrix $\Sigma$ and add a Gaussian term $N(0,\Sigma)$.  Then the set of parameters for the model $X+N(0,\Sigma)$ is $(\bp, V, \Sigma)$ and the probability density function is
\begin{equation}
\rho(x) = \E_{X\mid p, V} \rho_\Sigma(x-X),
\end{equation}
which exists because $\rho_\Sigma(x-X)$, the probability density function of $N(0,\Sigma)$, is a bounded function of $X$.

The log-likelihood of parameters $(\bp, V, \Sigma)$ for the model $X+N(0,\Sigma)$ with respect to observed data $(x_i)_{i=1}^N$ lying in $\R^n$ is
\begin{equation}
L(\bp, V, \Sigma)
  = \log \rho((x_i)_{i=1}^N \mid p,V,\Sigma)
  = \sum_{i=1}^N \log\E_{X\mid p,V}\rho_\Sigma(X-x_i).
\end{equation}
In all but the simplest examples this is not a function that can be calculated exactly or easily estimated.  But as we will see it can be iteratively maximised using the EM algorithm.

\subsection{Examples of linearly embedded mixture models}
\label{section_lemm_examples}
Our principle example of a linearly embedded mixture model is a simplicial mixture model, where the indexing set $A$ is a set $A_k(m)$ of simplices, and the distributions $U_S$ are described in Section~\ref{section_smm}.

The simplest example is when $A=\{1,\cdots,m\}$ and each~$U_i$ is the constant distribution with value the basis vector~$e_i\in \R^m$.  Then the linear mixture model is a discrete distribution with values the columns of $V$.  Fitting such a model is equivalent to fitting a Gaussian mixture model where the covariance matrices for the points are all equal.

When $A$ is a singleton set the mixture is trivial, consisting of a single distribution~$U$.  For example if $m < n$ and $U$ is the unit multivariate Gaussian distribution on~$\R^m$, then $X=VU$ is a multivariate Gaussian distribution with support a subspace of $\R^n$ of dimension $m$ or less.
Fitting a distribution $X + N(0, \sigma I)$ is equivalent to a form of principal component analysis where $X$ picks out the subspace spanned by the $m$ eigenvectors of the covariance matrix of the data corresponding to its $m$ greatest eigenvalues.

Another example is a generative topographic mapping (GTM)~\cite{bishop1998gtm}.  Pick $m$ points $y_1,\cdots,y_m$ in a metric space $M$ which is equipped with a probability measure.  For example $M$ could be the unit square and the $m$ points could form the corners of a grid within the square.
Let $a_i(y) = \exp(-\tfrac1{2s}d(y, y_i))$ for $i=1,\cdots,m$ and a fixed scaling parameter $s$.  Then define a map $\phi$ from $M$ into $\R^m$ by sending $y\in M$ to $\phi(y)\in \R^m$ with $\phi(y)_i = a_i(y) / \sum_{i=1}^m a_i(y)$.  Let $U$ be the distribution defined by $\phi(M)$.
Then fitting the associated linear mixture model $X=VU=V\phi(M)$ to data is equivalent to specifying vertex positions for each $i=1,\cdots,m$ and extending this to a map into~$\R^n$ taking $y$ to $\sum \phi(y)_i v_i$.
This is exactly the GTM of~\cite{bishop1998gtm}, but instead of their approximation of the EM algorithm, the methods below define a stochastic variant.

\subsection{Maximum likelihood fitting and the EM algorithm}
\label{section_lemm_EM}
The EM algorithm is an iterative approach to maximising log-likelihood, we follow the treatment of~\cite{bishop2006pattern}, Section~9.3.
Each step of the EM algorithm updates the parameters via the formula $\theta^\new = \argmax_{\theta} Q(\theta,\theta^\old)$ for a function $Q$ that we will define next. The log-likelihood of $\theta^\new$ is provably greater than the log-likelihood of $\theta^\old$.

Suppose that $\thetaold=(\bp^\old,V^\old,\Sigma^\old)$ is a set of parameters and that $\bX=(x_i)$ is our observed data.  For each data point $x_i$ and set of parameters $\theta$ there is a posterior distribution on the latent variables $C_i, Z_i$, we write $\bZ=(Z_i)$ and $\bC=(C_i)$.
The expectation of the log-likelihood function computed over the posterior distribution $\bZ$ given the parameters $\thetaold$,
\begin{equation}\label{eq_Qdefn}
  Q(\theta, \thetaold) =
    \E_{\bZ,\bC\mid \bX,\thetaold} \log \rho(\bX,\bZ,\bC \mid \theta)
\end{equation}
is maximised over $\theta$, to give a new set of parameters~$\theta^\new$ with increased log-likelihood.  However this relies on the existence/choice of probability density function for the joint distribution $(\bX,\bZ,\bC)$ which depends on the choice of measure for each~$U_S$.
Assuming the existence we could split the function using Bayes' Law
\begin{equation}
  \log \rho(\bX,\bZ,\bC \mid \theta) =
    \log\rho(\bX\mid\bZ,\bC,\theta) +
    \log\rho(\bZ \mid \bC, \theta) +
    \log P(\bC \mid \theta)
\end{equation}
The first and last terms exist, we write out their formulae later.  However the middle term is ambiguous as $U_S$ will not typically have a density function with respect to the Lesbegue measure on~$\R^m$.  However given $C$ the probability of $Z$ does not depend on the parameters~$\theta$.  So $\log\rho(\bZ\mid\bC,\theta) = \log\rho(\bZ\mid\bC)$ and this term would not participate in the maximisation.
This means that we can choose any measure on each $U_S$ and it does not change the maximisation of $Q$. In particular we can choose the measure on $U_S$ such that $\rho_{U_S} \equiv 1$, so the term vanishes entirely.

The probability $P(C = S \mid \theta)$ is equal to the parameter $p_S$. The distribution of $X$ given $Z, C$ and parameters $\theta$ is a multivariate Gaussian on $\R^n$ with mean $VZ$ and covariance~$\Sigma$.
Hence the formula for $Q$ becomes
\begin{equation}
  Q(\theta,\thetaold) = \sum_{i=1}^N
     \E_i\Bigl(
    -\tfrac{n}2 \log(2\pi) - \tfrac12\log|\Sigma|
    -\tfrac12 (x_i-VZ_i)^t\Sigma^{-1}(x_i-VZ_i)
    +\log p_{C_i} \Bigr)
\end{equation}
where we have written $\E_i$ for $\E_{Z_i,C_i\mid x_i,\thetaold}$.
Define the following matrices
\begin{align}
  Q_{ZZ} &= \sum_{i=1}^N \E_i Z_iZ_i^t \label{eq_QZZ}\\
  Q_{ZX} &= \sum_{i=1}^N \E_i Z_ix_i^t \label{eq_QZX}\\
  Q_{XX} &= \sum_{i=1}^N x_ix_i^t.     \label{eq_QXX}
\end{align}
and let $q_S = \sum_{i=1}^N P(C_i=S \mid x_i)$.
The matrices have dimensions $m\times m$, $m\times n$ and $n\times n$ respectively.
Then the formula for $Q(\theta,\thetaold)$ may be rewritten
\begin{equation}
  \text{const}
  +\tfrac{N}2\log|\Sigma^{-1}|
  -\tfrac12 Q_{XX} : \Sigma^{-1}
  + Q_{ZX} : V^t\Sigma^{-1}
  -\tfrac12 Q_{ZZ} : V^t\Sigma^{-1}V
  + \sum_{S\in A}q_S\log p_S,
\end{equation}
where for matrices $A,B$ of the same shape, $A : B$ is the matrix contraction $\sum_{ij}A_{ij}B_{ij}$.
Maximising $Q(\theta,\thetaold)$ with respect to $\theta$ is now straightforward.  For the probabilities we differentiate $Q$ with respect to each $p_S$, then subject to the constraint $\sum_{S\in A}p_S = 1$, we find
\begin{equation}\label{eq_pnew}
  p_S^\new = q_S/(\sum_S q_S).
\end{equation}
Differentiating with respect to $V$ and equating to~0 we obtain the matrix equation
\begin{equation}
  Q_{ZX}\Sigma^{-1} - Q_{ZZ}V^\new\Sigma^{-1} = 0,
\end{equation}
which, after cancelling $\Sigma^{-1}$ and assuming that $Q_{ZZ}$ is invertible, yields
\begin{equation}\label{eq_Vnew}
V^\new = Q_{ZZ}^{-1}Q_{ZX}.
\end{equation}
Finally, differentiating with respect to $\Sigma^{-1}$ (recall that the derivative of $\log|\Sigma^{-1}|$ is~$\Sigma$) and rearranging, we obtain
\begin{equation}\label{eq_Sigmanew}
  \Sigma^\new = \frac1{2N}\left((V^{\new})^tQ_{ZZ}V^\new - 2(V^{\new})^tQ_{ZX} + Q_{XX}\right).
\end{equation}
If we restrict $\Sigma$ to being isotropic, i.e. $\Sigma=\sigma I$, or diagonal with coefficients $\sigma_1,\cdots,\sigma_n$, then the respective formulae are
\begin{equation}
  \sigma^\new = \frac1n \text{tr}(\Sigma^\new)
    \quad\text{ and }\quad
  \sigma_i^\new = \Sigma_{ii}^\new,
\end{equation}
where $\Sigma^\new$ is given in~\eqref{eq_Sigmanew}.

\subsection{Calculating or estimating the expected $q$-values}
The matrix $Q_{XX}$ depends only on the dataset so needs only calculating once, whereas the values $q_S$, $Q_{ZZ}$ and $Q_{ZX}$ must be calculated at each step of the EM algorithm. The matrices are sums~\eqref{eq_QZZ}, \eqref{eq_QZX} over the data $x_i$ of certain expected values over the posterior distribution of $C_i, Z_i$ given $x_i$ and parameters~$\thetaold$.
The value $q_S$ is a sum over the data $x_i$ of the expected values $q_{S,i}=\E_{U_S}\rho_\Sigma(x_i-Vu)$ this time over the prior distribution on~$U_S$.
Recall that the prior distribution on~$Z$ is a mixture of distributions $U_S$, so expected values can be expanded as
\begin{equation}
  \E_{Z\mid x_i} f(z) = \sum_{S} p_S \E_{U_S\mid x_i} f(u).
\end{equation}
Then by Bayes' Law we have
\begin{equation}
  \E_{U_S\mid x_i} f(u) =
    \frac1{q_{S,i}}
    \int f(u)\rho_\Sigma(x_i-Vu) \mu_{U_S}(u),
\end{equation}
where we have expressed the conditional density function as $\rho(x\mid C=S, U_S=u) = \rho_\Sigma(x-Vu)$.
Note that with the functions $f(u)$ used to calculate $Q_{ZZ}$ and $Q_{ZX}$, these matrices can be expressed in terms of the central moments of the posterior distribution on~$U$.
The classes of distributions $U$ for which these expected values may be computed explicitly are limited, but they include Gaussian distributions (including those with support on affine planes) on $\R^m$, finite discrete distributions, and uniform distributions on line segments.
In the case of a uniform distribution on a line segment, the posterior distributions are uni-variate truncated Gaussians embedded in $\R^m$.

Performing the calculation of $q_S$, $Q_{ZZ}$ and $Q_{ZX}$ directly requires summing over every distribution $U_S$ for $S\in A$ and every datapoint $x_i$, which is slow if the cardinality of $A$ is large.
Fortunately the expected values can instead be estimated using the Metropolis algorithm once one has constructed a Markov chain defined on $C$ and $Z$ with the appropriate stationary distribution, which is a simple task for most distributions $U_S$.  See Appendix~\ref{section_mcmc} and Table~\ref{table_state_diagram} for further details.

If the expected values are only estimated and the maximisation step is performed using the estimates then this is a stochastic variant of the EM algorithm and we no longer have guarantees of monotonicly increasing log-likelihood or of convergence.
An advantage of using estimates is one of efficiency: in many applications the posterior distributions $C_i$ will have low entropy, so for any given datapoint $x_i$ most of the component distributions $U_S$ will have little probability mass in the posterior distribution.  The Markov chain Monte Carlo approach means that these components are only rarely sampled.

\subsection{Model comparison via minimal description length}
\label{section_encodingrate}
The EM algorithm maximises the log-likelihood, but calculating or estimating the log-likelihood is usually not feasible.  To compare different sets of parameters from the same model, or to compare parameters between different types of models, we defined the intrinsic encoding rate~\eqref{eq_encodingrate} of a simplicial mixture model.  We will now explain the derivation of the formula and the terms in the expression in more detail.  The general version for a linearly embedded mixture model is
\begin{equation}\label{eq_encodingrate2}
  h_R(\bp, V, \Sigma) = H(\bp) +
    \sum_{S\in A} p_S R_{VU_S}(\Sigma) + \frac12\log\left[(4\pi e)^n\det(\Sigma)\right].
\end{equation}
This is a measure of the average amount of information required to encode values of the model $X+N(0,\Sigma)$ by first encoding a value~$S\in C$, then a value of $z\in Z$ given $C=S$, before finally encoding a value $x = (x-Vz) + Vz$.
The distribution~$C$ is discrete and has entropy $H(\bp)$.
The random variable $X+N(0,\Sigma)$ is continuous and has well-defined density function, so the differential entropy is the appropriate measure of information.  However these properties can not be assumed for $Z$ and so the differential entropy cannot be used, instead we use rate distortion theory and allow $Z$ to be quantised.  To be precise, in applying rate distortion theory we do not quantise the single variable $Z$, but instead a set of i.i.d. variables drawn from~$Z$.

See Chapter~10 of~\cite{cover2012elements} for a full introduction to rate distortion theory.
A distortion function $d$ is first chosen to measure the distance between a random variable $U$ and a given finite representation~$\Uhat$.  Then the rate distortion function $R(D)$ specifies the minimum encoding rate achievable by a finite representation of the random variable, such that the expected distortion is less than the given~$D$.

For our application we choose a distortion function in such a way that the covariance matrix of the difference $VU - V\Uhat$ is dominated by~$\Sigma$.
If we define
\begin{equation}\label{eq_distortion}
  d(u, \widehat{u}) = (Vu - V\widehat{u})^t\Sigma^{-1}(Vu - V\widehat{u})
\end{equation}
then $R(1)$ is the minimum encoding rate achievable by a finite representation while satisfying the above covariance condition.
We write $R_{VU}(\Sigma) = R(1)$ to emphasise the dependence on $V$, $U$ and~$\Sigma$.  Then $Z$ can be encoded at an average rate of $\sum_S p_SR_{VU_S}(\Sigma)$ with covariance matrix of $VZ-V\Zhat$ dominated by~$\Sigma$.
With the difference $VZ-V\Zhat$ having covariance dominated by $\Sigma$, we know that the difference $Y = VZ+N(0,\Sigma) - V\Zhat$ has covariance dominated by~$2\Sigma$.
The difference distribution $Y$ is not necessarily a multivariate Gaussian, but by the maximum entropy property its differential entropy is bounded above by the entropy of $N(0, 2\Sigma)$ which is the final term of the intrinsic encoding rate.
The above derivation implies that $h_R(\bp, V, \Sigma) \geq h$ where $h$ is the differential entropy of $X + N(0, \Sigma)$.

\subsection{Discussion of the intrinsic encoding rate}
In practice computing the rate distortion functions $R_{VU}(\Sigma)$ for the component distributions $U_S$ may not be feasible.  A simple but effective alternative is to compute the rate distortion functions for the normal distribution with the same mean and covariance matrix as $U_S$, this is the method we implemented and applied in the results of Figure~\ref{fig_10digits}.

A small intrinsic encoding rate suggests a low entropy for $C$, so the mixture is concentrated on fewer components.
It suggests a low expected rate distortion which means that weighted sum of the `volumes' of the embedded mixture components $VU_S$ should be small, which discourages overlaps and discourages empty regions of the distributions far away from the data.
Finally it suggests that the determinant of the covariance should be small, so the distances squared between the model and the data should be small.
These three properties agree with intuition about what implies a good fit of the model to the data.
Ofcourse they are traded off against each other.

The intrinsic encoding rate $h_R$ is a function of parameters and does not directly depend on the data~$(x_i)_{i=1}^N$.
This can only be an appropriate measure of fit when the model parameters $(\bp, V, \Sigma)$ are stationary under the EM algorithm as applied to the dataset.
Define distributions $(C', Z')$ as follows: pick a datapoint $x_i$ uniformly from the dataset, then pick $(c, z)$ from the conditional distribution of $(C,Z)$ given $x_i$.  Under the assumption on the parameters, $C'$ is equal in distribution to~$C$.
The distribution $Z'$ can not be assumed to be equal to $Z$, however from the definition of $V$ and $\Sigma$ in the maximisation step of the EM algorithm we do know that the average covariance matrix of $VZ' - x_i$ over the dataset is equal to~$\Sigma$.
Since we do not have any concrete results about the differences in distribution between $Z$ and $Z'$ we cannot offer any formal connection between $h_R$ and the log-likelihood of the data.  This is an opportunity for further empirical research.

\section{A Markov chain Monte Carlo expectation maximisation algorithm}
\label{section_mcmc}
In Appendix~\ref{section_LEMM} the EM algorithm for a linearly embedded mixture model was derived. Each iteration is divided into two steps, the calculation of the $q$-values, $q_S$, $Q_{ZZ}$ and $Q_{ZX}$ and then the minimisation of the function $Q$ defined in~\eqref{eq_Qdefn} using the formulae~\eqref{eq_pnew},~\eqref{eq_Vnew} and~\eqref{eq_Sigmanew}.
The $q$-values are computed as a sum over the data~$x_i$, with each term an expectation over the posterior distribution of the latent variables $C$ and $Z$ given that $X+N(0,\Sigma)=x_i$.  Explicit formulae can be used only when the distributions $U_S$ are of a certain type, for example explicit formulae were used for fitting models for Figures~\ref{fig_10digits} and~\ref{fig_insideno9} since each variable $U_S$ was a uniform distribution on a line segment.

We can estimate the $q$-values if we can generate samples from the posterior distributions.  Fortunately a Markov chain approach allows for efficient sampling.  In this approach there is a Markov chain for each datapoint~$x_i$, whose state can described as a pair $z_i=(S_i, u_i)$ where $S_i\in A$ determines a choice of distribution and $u_i$ is in the support of~$U_{S_i}$.
The state is updated using the Metropolis-Hastings algorithm, for this we require a Markov chain update $M$ with stationary distribution~$Z$ to propose changes to each of the datapoint Markov chains.  So let $z'_i = (S'_i, u'_i) = M(S_i, u_i)$ be the proposed update.  We accept this update with probability
\begin{equation}
  \max\left(1,
  \frac{p(z'_i)}{p(z_i)}
  \right),
\end{equation}
where $p(z_i) = \rho_\Sigma(x_i - Vu_i)$ is the conditional probability of $z_i$ given $x_i$ and parameters $(\bp, V, \Sigma)$.
The choice of $M$ depends on the set $A$ and the distributions $U_S$.  If one has Markov chain updates $M_A$ and $M_S$ such that $M_A$ has unique stationary distribution~$\bp$ on~$A$ and each~$M_S$ the stationary distribution $U_S$, then half the time one could update the second coordinate, i.e. send $(S,u)$ to $(S, M_S(u))$ and half the time the first coordinate, so send $(S,u)$ to $(M_A(S), u')$ where $u'$ is drawn from~$U_{M_A(S)}$.

\subsection{Implementation notes for the stochastic EM algorithm}
\label{section_implementation}
For the full details of the implementation one can read the documented code~\cite{griffin2019smm}.  We will explain the core concepts, however optimisations through caching of values are not discussed.  The state at any stage of the algorithm and the four possible actions that change the state are listed in Table~\ref{table_state_diagram}.

\begin{table*}\centering
  \renewcommand{\arraystretch}{1.2}
  \caption{The state of the stochastic Markov chain EM algorithm and the possible actions on the state.}
  \begin{tabular}{@{}lp{0.8\linewidth}@{}}
    \textbf{State} & \\
    \toprule
    Parameters &
    $\theta=(\bp=(p_S)_{S=1}^M, V, \Sigma$) \\
    Simulation &
    $z_i = (S_i, u_i)$ for $i=1,\cdots,N$ \\
    Q-Values   &
    $\widehat{Q}_{ZZ}$, $\widehat{Q}_{ZX}$ and $\widehat{q}_S$ for $S=1,\cdots,M$\\
    \bottomrule
    & \\
    \textbf{Actions} & \\
    \toprule
    M-step    &   (\textit{maximisation})
      the parameters are updated using the Q-value estimates and the Q-values are reset to~$0$.  The simulation state is not changed.\\
    Q-step    &   (\textit{contribution to Q-values})
      the parameters and simulation state are used to add a term $\sum_i z_iz_i^t$ to $\widehat{Q}_{ZZ}$, a term $\sum_i z_ix_i^t$ to $\widehat{Q}_{ZX}$ and a term $|\{S_i = S\}|$ to $\widehat{q}_S$.\\
    C-step    &   (\textit{C-changing Markov step})
      for each $i$ a new candidate for $z_i$ is chosen according to $P(Z\mid \theta)$ and either accepted or rejected according to the Metropolis-Hastings criterion.\\
    U-step    &   (\textit{U-changing Markov step})
      for each $i$ the choice of component distribution~$S_i$ is retained but a candidate is chosen from the given Markov chain on $U_{S_i}$, then the Metropolis-Hastings criterion is applied.\\
    \bottomrule
  \end{tabular}
  \label{table_state_diagram}
\end{table*}

To carry out the algorithm the state is first initialised.
Typically the columns of~$V$ will be chosen at random from the data~$X$, the probabilities~$\bp$ could be chosen to be uniform and the covariance matrix chosen to be the covariance of the dataset~$X$. The simulation state is initialised by drawing from each latent variable of the model with the initial parameters $(\bp, V, \Sigma)$. The Q-values are initialised to zero.

After initialisation the four actions are carried out according to some regime: we would like to apply sufficiently many simulation steps (C and U) so that the simulation state reaches the stationary distribution of~$Z$ given~$x_i$, then more simulation steps and Q-steps must be carried out so that a reasonable estimate of the Q-values are obtained, only then should a maximisation (M) step be carried out.
However if we assume each maximisation step only changes the posterior distributions on the latent variables slightly, then fewer steps (C and U) need to be carried out to reach the stationary distribution.

For the fitting of the models of Figure~\ref{fig_unmixing} the regime used was
\begin{equation}\label{eq_regime}
  (CUQM)^{2000}((CUQ)^5M)^{1000}
\end{equation}
where the actions are carried out by reading the word left to right.
In this case the number of datapoints was large at 32041 considering the data was only three dimensional.  So with many parallel Markov chains it was judged that few steps per maximisation were required.
In this case we specified the regime in advance, but one could perform a test for convergence after every M-step and use this to terminate the algorithm.

The same code may be used to perform estimations of values over the conditional distributions of the latent variables.  For example the results of Figure~\ref{fig_unmixing} are obtained by computing the expected value of~$Z$ given every pixel in the image.  After fitting the model using the stochastic EM algorithm, more simulation steps are performed whilst calculating an average of the~$Z$ values for each data point.

\subsection{Discussion of implementation}
The stochastic EM algorithm is highly parallelisable.  The simulation state and the computations of the Q, C and U steps may all be be carried out in parallel across different CPUs/GPUs.  Performing an M step requires the accumulation of the Q-values and some simple linear algebra after which only the parameters are sent back to the individual processors.

In our implementation we make the assumption that every $U_S$ on $\R^m$ arises from a common distribution $U$ on $\R^k$, but with a different linear embedding $M_S:\R^k\rightarrow\R^m$ for each~$S\in A$.  This means that the generation of the U-step candidates are independent of the values of~$C$.

In the C-step the latent discrete variables are changed by sampling afresh from~$C$.  However if $C$ has extra combinatorial structure then different updates could be used.  For example to change a $k$-simplex $i_0\leq\cdots\leq i_k$, one could choose an index at random, remove it and then insert a new index chosen at random.  With a smaller change we would expect a higher acceptance rate when applying the Metropolis-Hastings criterion.

For many applications it is expected that~$\bp$ would have low entropy compared to the size~$M$ of its support, i.e. $2^{H(\bp)} \ll M$.  In this case the values $\widehat{q}_S$ will be sparse.  Instead of specifying~$\bp$ directly, we could replace it by a Dirichlet distribution with parameters given by~$1$ plus a sparse vector.
Whereas the actions and state of the current algorithm scale in complexity and memory use with $M$, one may hope that with such changes it would scale with $2^{H(\bp)}$.

\end{document}